\documentclass[11pt,letterpaper]{amsart}
\usepackage[margin=1in]{geometry}
\usepackage{amssymb,stmaryrd}
\usepackage[foot]{amsaddr}
\usepackage{hyperref}
\usepackage{color}
\usepackage{curves}
\usepackage[normalem]{ulem} 

\newtheorem{theorem}{Theorem}[section]
\newtheorem*{theorem*}{Theorem}
\newtheorem{corollary}[theorem]{Corollary}
\newtheorem{proposition}[theorem]{Proposition}
\newtheorem{lemma}[theorem]{Lemma}

\theoremstyle{definition}
\newtheorem{definition}[theorem]{Definition}
\newtheorem{example}[theorem]{Example}
\newtheorem{remark}[theorem]{Remark}

\newtheorem{question}[theorem]{Question}

\DeclareMathOperator{\N}{N}
\DeclareMathOperator{\V}{V}
\DeclareMathOperator{\C}{C}
\DeclareMathOperator{\K}{K}
\DeclareMathOperator{\dist}{d}
\DeclareMathOperator{\PSL}{PSL}
\DeclareMathOperator{\PGL}{PGL}
\DeclareMathOperator{\Sym}{Sym}

\begin{document}

\title{Parameters for certain locally-regular graphs}
\author{Marston Conder, Jeroen Schillewaert and Gabriel Verret}
\address{{Department of Mathematics, University of Auckland, 38 Princes Street, 1010 Auckland, New Zealand. AMS code: 05C99; Keywords: locally regular graphs}} 
\email{m.conder@auckland.ac.nz, j.schillewaert@auckland.ac.nz, g.verret@auckland.ac.nz}
\date{}

\maketitle

\begin{abstract}
A graph is called $(k,t)$-regular if it is $k$-regular and the induced subgraph on the neighbourhood of every vertex is $t$-regular. We find new conditions on $(k,t)$ for the existence of such graphs
and provide a wide range of examples.

\end{abstract}

\section{Introduction}

Regular graphs are ubiquitous and well-studied objects in combinatorics. A slightly unusual way to define regularity goes as follows. Given a graph $G$ and a vertex $v$ of $G$, the \emph{local graph at $v$} is the subgraph of $G$ induced on the neighbourhood  $\N(v)$ of $v$, and is denoted by $G_v$. A simple graph is then \emph{$k$-regular} if all its local graphs have order $k$. 
(Note: All graphs in this paper are simple. The graphs may be infinite, but the valency $k$ will always be finite.)
 
The above concepts can be taken further quite naturally by asking that the local graphs themselves be regular, and that all have the same valency. 
A graph is \emph{locally-regular} all its local graphs are regular, and is $(k,t)$-\emph{regular} if it is $k$-regular and all its local graphs are $t$-regular.   
Such graphs are the topic of our paper. More specifically, we would like to address the following question. 

\begin{question}\label{BigQuestion}
For which pairs $(k,t)$ does there exist a $(k,t)$-regular graph?
\end{question}

Our interest was drawn to this question (and more generally, to ``$(a_1,\dots,a_n)$-regular'' graphs)  following some recent work on expander graphs and PCP theory, in \cite{CLP} followed by \cite{CLST}. It is a very natural and basic question, and yet does not seem to be trivial and not much seems to be known about it. 
In fact, the only contribution to this topic prior to~\cite{CLP} that we are 
aware of was by Zelinka~\cite{Zelinka}, who pointed out some basic necessary conditions 
for existence of $(k,t)$-regular graphs, constructed a large range of examples using some basic graph products, and then ruled out the existence of $(7,4)$-regular graphs in a somewhat \emph{ad hoc} manner.

In Section \ref{sec:BS} we collect some basic facts which we then use in Section \ref{sec:fix-t-large-k} to obtain new information about $(k,t)$-regular graphs for large $k$ with respect to $t$,  slightly improving what was achieved in \cite{Zelinka}. 
In Section \ref{sec:FC} we present our main findings, namely for large  $k$ and $t$ with respect to $k-t$, 
and in particular, we prove the following theorem, which shows there are no $(k,t)$-regular graphs for infinitely many such pairs, including $(7,4)$.

\begin{theorem}\label{prop:diagonal}
Let $d\geq 1$ and let $G$ be a connected $(k,k-d)$-regular graph. 
If $k\geq d^2-d+1$, then $d$ divides $k$, and $G$ is isomorphic to the complete multipartite graph $\K_{\frac{k+d}{d}}[\overline{\K_{d}}]$.  
\end{theorem}

Next, in
Section \ref{sec:smallcases}, we summarise our knowledge regarding Question~\ref{BigQuestion} for small values of $k$ and $t$, and we conclude the paper in Section \ref{sec:conclusion} by discussing possible future avenues of research.

\section{Basic necessary conditions and constructions}\label{sec:BS}

Let $G$ be a $(k,t)$-regular graph,  
so that all of its local graphs have order $k$ and are $t$-regular. Yet another way of expressing this is that every vertex is contained in $k$ edges and every edge is contained in $t$ triangles ($3$-cycles).

Note that a graph is $(k,t)$-regular if and only if all its connected components are $(k,t)$-regular. One can therefore restrict oneself to the connected case without loss of generality, which we often do.

Next, we present a few obvious necessary conditions, some of which were already observed by Zelinka in~\cite{Zelinka}. To avoid trivialities  when considering $(k,t)$-regular graphs, from now on we will always assume that $k\geq 1$.

\begin{proposition}\label{Prop:Basic}
If there exists a $(k,t)$-regular graph, then $k\geq t+1$ and $kt$ is even, so at least one of $k$ and $t$ is even.
\end{proposition}
\begin{proof}
Since the local graphs have order $k$ and valency $t$, we must have $k\geq t+1$, while the second part follows from the Handshaking Lemma. 
\end{proof}

Next, we observe how three well-known graph products can used to produce new examples from old ones. See~\cite{IK} for definitions of these graph products. 
The proofs are all very straightforward, so are omitted.

\begin{proposition}\label{Prop:CartesianProduct}
If $G$ is $(k,t)$-regular and $H$ is $(\ell,t)$-regular, then the Cartesian product $G \,\square\, H$ is $(k+\ell,t)$-regular. 
\end{proposition}

\begin{proposition}\label{Prop:LexProduct}
If $G$ is $(k,t)$-regular, then the lexicographic product $G[\overline{\K_m}]$ of $G$ with the edgeless graph $\overline{\K_m}$ is $(km,tm)$-regular.
\end{proposition}

\begin{proposition}\label{Prop:TensorProduct}
If $G$ is $(k,s)$-regular and $H$ is $(\ell,t)$-regular, then the tensor product $G \times H$ is $(k\ell,st)$-regular.   
\end{proposition}

\section{Large  $k$ with respect to $t$}\label{sec:fix-t-large-k}

\begin{example}\label{Ex:Complete}
The complete graph $\K_{k+1}$ is $(k,k-1)$-regular for all $k \ge 2$.
\end{example}

In fact it is easy to see that $\K_{k+1}$ is the only connected  $(k,k-1)$-regular graph; this also follows from Theorem~\ref{prop:diagonal}.

Next, one can deal with a surprisingly large number of pairs $(k,t)$ simply by using only the complete graphs and iterating the product operations from Section~\ref{sec:BS}. 
In fact, for fixed $t$ this produces $(k,t)$-regular graphs for all sufficiently large $k$, apart from those forbidden by the parity condition in Proposition~\ref{Prop:CartesianProduct}, as we now show. 
(Note: The bounds we obtain are slight improvements over those in \cite{Zelinka}, and the proofs are also more streamlined.)
\smallskip

\begin{proposition}\label{Prop:EvenT}
If $t\geq 0$ is even and $k\geq t^2+t$, then there exists a $(k,t)$-regular graph.
\end{proposition}

\begin{proof}
By Example~\ref{Ex:Complete}, there exist $(t+1,t)-$ and $(\frac{t}{2}+1,\frac{t}{2})$-regular graphs. Using Proposition~\ref{Prop:LexProduct}, there exists a $(t+2,t)$-regular graph. It now follows from Proposition~\ref{Prop:CartesianProduct} that if $k$ can be written as a linear combination of $t+1$ and $t+2$ with positive integer coefficients, 
then there exists a $(k,t)$-regular graph. Finally, since $\gcd(t+1,t+2)=1$,  the solution to the Frobenius coin problem \cite{NW} tells us that this happens for all $k\geq (t+1)(t+2)-(t+1)-(t+2)+1=t^2+t$.   
\end{proof}
\vskip 4pt

\begin{proposition}
If $t\geq 1$ is odd and  $k\geq (t-1)^2$ is even, then there exists a $(k,t)$-regular graph. 
\end{proposition}

\begin{proof}
By Example~\ref{Ex:Complete}, there exists a $(t+1,t)$-regular graph, and also there exists a $(2,1)$-regular graph. Using Proposition~\ref{Prop:LexProduct}, this implies that there exists a $(2t,t)$-regular graph. Note also that $\gcd(t+1,2t)=2$, and hence that every sufficiently large even integer can be expressed as a linear combination of $t+1$ and $2t$ with positive coefficients. In fact, by the solution to the Frobenius coin problem, this happens for all $k\geq (t-1)^2$.  
\end{proof}
\vskip 4pt

\begin{remark}
The above two propositions do not cover all cases that can be obtained using a combination of Example~\ref{Ex:Complete} and the constructions from Section~\ref{sec:BS}.
For example, 
$\K_4$ is a $(3,2)$-regular graph, and so by Proposition~\ref{Prop:LexProduct} there exists a $(6,4)$-regular graph,   
and then this produces $(k,4)$-regular graphs for some other pairs of the form $(k,4)$ not already covered. 
\end{remark}

\section{Large  $k$ and $t$ with respect to $k-t$}\label{sec:FC}

Using `$-$' to denote set difference, we first give the following:

\begin{lemma}\label{lemma:Venn}
If $G$ is a $(k,t)$-regular graph and $\{x,y,z\}$ is a triangle in $G$, then  
$$|\N(x)-(\N(y)\cup \N(z))|=|\N(z)-(\N(x)\cup \N(y))|=|\N(y)-(\N(x)\cup \N(z))|.$$ 
\end{lemma}
\begin{proof}
Note that $\N(x)$, $\N(y)$ and $\N(z)$ have the same size $k$, and their pairwise intersections have the same size $t$, 
so the conclusion follows  by considering a Venn diagram for $\N(x) \cup \N(y) \cup \N(z)$. 
\end{proof} 

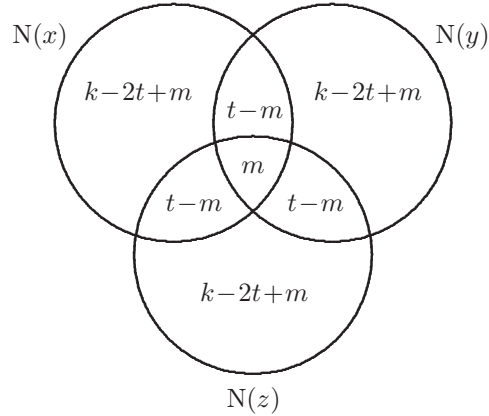
$$
\begin{picture}(0,150)(20,-92)
{\linethickness{0.7pt}
\put(-30,20){\bigcircle{90}} 
\put(30,20){\bigcircle{90}} 
\put(0,-30){\bigcircle{90}} 
\put(-4.5,1){\makebox(0,0)[bl]{\small $m$}}
\put(-10,21){\makebox(0,0)[bl]{\small $t\!-\!m$}}
\put(-33,-14){\makebox(0,0)[bl]{\small $t\!-\!m$}}
\put(13,-14){\makebox(0,0)[bl]{\small $t\!-\!m$}}
\put(-64,28){\makebox(0,0)[bl]{\small $k\!-\!2t\!+\!m$}}
\put(23,28){\makebox(0,0)[bl]{\small $k\!-\!2t\!+\!m$}}
\put(-20,-48){\makebox(0,0)[bl]{\small $k\!-\!2t\!+\!m$}}
\put(-91,48){\makebox(0,0)[bl]{\small $\N(x)$}}
\put(69,48){\makebox(0,0)[bl]{\small $\N(y)$}}
\put(-10,-90){\makebox(0,0)[bl]{\small $\N(z)$}}

}
\end{picture}
$$

\begin{corollary}\label{cor:Venn2Restated}
If $G$ is a $(k,t)$-regular graph and $x,y,z$ are three distinct vertices of $G$ such that $y,z\in \N(x)$, 
and $\N(x)$ is a subset of $\N(y)\cup \N(z)$,
then \medskip 

\centerline{$\N(y)-\N(x)=\N(z)-\N(x), \  \ \N(x)-\N(y)=\N(z)-\N(y), \ \hbox{ and } \ \N(x)-\N(z)=\N(y)-\N(z).$} 

\end{corollary}
\begin{proof}
First, $y \in \N(x)$ but $\N(x) - (\N(y) \cup \N(z)) = \emptyset$ and $y \not\in \N(y)$, 
so $y \in ((\N(x) \cap \N(z)  - \N(y))$. In particular, $y \in \N(z)$, and therefore $\{x,y,z\}$ is a triangle in $G$. 
Hence by Lemma~\ref{lemma:Venn}, we find that $\N(y)$ is a subset of $\N(x)\cup \N(z)$, and $\N(z)$ is a subset of $\N(x)\cup \N(y)$, 
and the conclusion again follows from a Venn diagram for $\N(x) \cup \N(y) \cup \N(z)$.  
\end{proof}

Note that the hypothesis that $\N(x) \subseteq \N(y)\cup \N(z)$ is equivalent to there being no vertex 
in $\N(x)$ adjacent to both $y$ and $z$ in the complement $\overline{G_x}$, and hence also equivalent to $\dist_{\overline{G_x}}(y,z)\geq 3$ (where $\dist_{\overline{G_x}}$ denotes the distance function in $\overline{G_x}$). 
This leads us to introduce the following definition.

\begin{definition}
For a graph $G$, we denote by $G^{\geq 3}$ the graph on the same vertex-set as $G$, 
with two vertices $x$ and $y$ adjacent in $G^{\geq 3}$ if and only $\dist_G(x,y) \ge 3$. 
(Note that this includes the possibility that $\dist_G(x,y) = \infty$.)
Also we say that $G$ is \emph{far-connected} if $G^{\geq 3}$ is connected.
\end{definition}

For example, every disconnected graph is far-connected,  
and for $n \ge 3$ the cycle graph $\C_n$ is far-connected if and only if $n \ge 7$.

\begin{lemma}\label{CountingTwoPaths} 
Let $G$ be a connected finite $d$-regular graph with a partition of its vertex set into two sets $X$ and $Y$, 
and let  $\mathcal{P}$ be the set of undirected paths of length at most $2$ in $G$ between $X$ and $Y$. 
Then 
$$|\mathcal{P}|\leq \frac{|X||Y|}{|X|+|Y|}\left(d^2+d+\frac{1}{4}\right).$$

\end{lemma}
\begin{proof}
First, let $x_1,x_2,\dots,x_m$ and $y_1,y_2,\dots,y_n$ be the vertices of $X$ and $Y$ respectively.  
Next, for a vertex $x_i$ in $X$, denote by $a_i$ the number of edges in $X$ incident with $x_i$, 
and by $b_i$ the number of edges from $x_i$ to vertices in $Y$, so that $a_i+b_i = d$.  
Similarly, for a vertex $y_i$ in $Y$, denote by $c_i$ the number of edges in $Y$ incident with $y_i$. 
Also let $H$ be the bipartite subgraph of $G$ obtained by deleting all edges within $X$ or within $Y$, 
keeping only the edges between $X$ and $Y$.
We will show that $|\mathcal{P}|$ is maximised when $H$ is bi-regular, 
with all vertices in the same bipart ($X$ or $Y$) having the same valency in $H$.  

Suppose that $H$ is not bi-regular, because (say) $a_i<a_j$ and hence by $d$-regularity also $b_i>b_j$. 
Consider a vertex $x_k\in X - \{x_i\}$ that is adjacent to $x_j$ but not to $x_i$, and a vertex $y_\ell\in Y$ that is adjacent to $x_i$ but not to $x_j$. 
Now replace the edges $\{x_i,y_\ell\}$ and $\{x_k,x_j\}$ by the edges $\{x_i,x_k\}$ and $\{x_j,y_\ell\}$. 
Then the valencies in $G$ of all vertices are preserved, and the number of paths in $\mathcal{P}$ (of length at most two between $X$ and $Y$) 
that we lose is equal to $(1+a_i+c_\ell)+(b_j+b_k)$, while the number of such paths that we gain is equal to $(1+a_j+c_\ell)+(b_i+b_k)$, 
and hence the net gain is $a_j+b_i-(a_i+b_j) = (a_j-a_i) + (b_i-b_j) > 0$. 
Thus $|\mathcal{P}|$ is maximised when $H$ is bi-regular.

Now suppose that $H$ is bi-regular, with vertices from $X$ having valency $r$ (in $H$), 
and vertices from $Y$ having valency $s$ (in $H$), where $1 \leq r,s \leq d$. 
Then in $G$ the number of paths of length $1$ from $X$ to $Y$ is $mr=ns$, 
while the number of paths of length $2$ from $X$ to $Y$ is $m(d-r)r+n(d-s)s$, so 
$$|\mathcal{P}| = mr+m(d-r)r+n(d-s)s = mr+m(d-r)r+m(d-s)r = mr(1+2d-r-s).$$

Letting $u = r+s$, this gives
$$\frac{(m+n)|\mathcal{P}|}{mn} = \frac{|\mathcal{P}|}{m}+\frac{|\mathcal{P}|}{n} = r(1+2d-r-s)+s(1+2d-r-s) = u(1+2d-u),$$ 
and if we think of $d$ as being fixed, and $u$ as a variable, then this is a quadratic function 
that attains its maximum of $\frac{(2d+1)^2}{4}=d^2+d+\frac{1}{4}$ when $u=\frac{2d+1}{2}$. 

It follows that \quad $\frac{(|X|+|Y|)|\mathcal{P}|}{|X||Y|} = \frac{(m+n)|\mathcal{P}|}{mn} \leq d^2+d+\frac{1}{4}$, \quad as required. 
\end{proof}
\vskip 3pt 

\begin{corollary}\label{cor:far-connected}
A $d$-regular graph with order greater than $d^2+d$ is far-connected. 
\end{corollary}
\begin{proof}
Let $G$ be a $d$-regular graph that is not far-connected. 
Then we can partition its vertex set into two non-empty subsets, say $X$ and $Y$, such that $d(x,y)\leq 2$ for every $x\in X$ and $y\in Y$. This implies that $G$ must be finite.
Now let $\mathcal{P}$ be the set of undirected paths of length at most $2$ in $G$ between $X$ and $Y$.   
Then by Lemma~\ref{CountingTwoPaths}, we know that $|\mathcal{P}|\leq \frac{|X||Y|}{|X|+|Y|}\left(d^2+d+\frac{1}{4}\right)$. 
Moreover, because $d(x,y)\leq 2$ for all $(x,y) \in X \times Y$, we have $|X||Y|\leq |\mathcal{P}|$ and thus $|X|+|Y|\leq d^2+d+\frac{1}{4}$. 
But $|X|+|Y|$  is an integer, and so $|X|+|Y|\leq d^2+d$.
\end{proof}

The following example family shows that the above bound is the best possible for every valency $d \geq 1$.

\begin{example}\label{MDEConder}
For fixed $d\geq 1$, let  $V=\{v_i: i\in\{1,\ldots,d\}\}$ and $W=\{w_{ij}: i,j\in \{1,\ldots,d\}\}$, 
and define $G$ as the bipartite graph on the vertex set $V\cup W$ in which $v_i$ adjacent to $w_{ik}$ for every $k$, and $w_{ik}$ adjacent to $w_{jk}$ for every $i\neq j$. 
Then $G$ has order $d^2+d$ and is $d$-regular. By considering $2$-paths of the form $(v_i,w_{ik},w_{jk})$, however, 
we see that vertices from different parts of $G$ lie at distance at most $2$ from each other, and so $G$ is not far-connected. 
\end{example}

Next, we say that vertices $x$ and $y$ are \emph{twins} if they have the same neighbourhood, that is, if $\N(x)=\N(y)$.
Note that here we include the possibility that $x = y$, which is not always assumed in the definition of twins.

\begin{proposition}\label{prop:TwinCount}
Let $G$ be a $(k,t)$-regular graph and let $x$ be a vertex of $G$ such that $\overline{G_x}$ is far-connected. Then $x$ has exactly $k-t$ twins in $G$. 
\end{proposition}
\begin{proof}
We first show that $\N(y)-\N(x)=\N(z)-\N(x)$ for all $y,z\in \N(x)$.  
Because $\overline{G_x}$ is far-connected, there exists a sequence $(v_1,v_2,\ldots,v_n)$ from $v_1 = y$ to $v_n = z$ 
with $v_i\in \V(G_x)$ for all $i$, and $\dist_{\overline{G_x}}(v_i,v_{i+1})\geq 3$ for all $i < n$. By the comments 
following Corollary~\ref{cor:Venn2Restated}, we see that $\N(x)$ is a subset of $\N(v_i)\cup \N(v_{i+1})$ for all $i < n$, 
and hence by repeated application of Corollary~\ref{cor:Venn2Restated} it follows that 
$\N(y)-\N(x) = \N(v_1)-\N(x) = \N(v_2)-\N(x) = \dots = \N(v_n)-\N(x) = \N(z)-\N(x)$.

Now let $x'\in \N(y)-\N(x)$, for some $y \in \N(x)$. Then by the previous paragraph, $x' \in \N(z) - \N(x)$ for every $z \in \N(x)$, 
and so every $z \in \N(x)$ lies in $\N(x')$. Thus $\N(x) \subseteq \N(x')$, and then since $G$ is regular, $|\N(x)| = |\N(x')|$, 
so $\N(x) =  \N(x')$, and therefore $x'$ is a twin of $x$. 
Conversely, if $x'$ is a twin of $x$, then $y \in \N(x) = \N(x')$, but $x \not\in \N(x) = \N(x')$ and so $x' \not\in \N(x)$, 
and therefore $x' \in \N(y)-\N(x)$. 
It follows that the set of twins of $x$ is exactly $\N(y)-\N(x)$, which has size $k-t$.
\end{proof}

\begin{corollary}\label{cor:CharCompleteMulti}
Let $G$ be a connected $(k,t)$-regular graph such that $\overline{G_x}$ is far-connected for all $x\in \V(G)$. 
Then $k-t$ divides $t$, and hence also divides $k$, and $G$ is isomorphic to the complete multipartite graph $\K_{\frac{2k-t}{k-t}}[\overline{\K_{k-t}}]$.  
\end{corollary}
\begin{proof}
By Proposition~\ref{prop:TwinCount}, every vertex of $G$ has exactly $k-t$ twins. 
It follows easily that $G$ is isomorphic to the lexicographic product $H[\overline{\K_{k-t}}]$ for some connected $\left(\frac{k}{k-t},\frac{t}{k-t}\right)$-regular graph $H$. In particular, $k-t$ divides $k$ and $t$. 
But also $\frac{k}{k-t} = \frac{t+k-t}{k-t} = \frac{t}{k-t}+1$, 
and hence $H$ is a complete graph with valency $\frac{k}{k-t}$ and order $\frac{k}{k-t}+1 =  \frac{k+k-t}{k-t} =  \frac{2k-t}{k-t}$. 
Thus $H$ is isomorphic to $\K_{\frac{2k-t}{k-t}}$.
\end{proof}
\vskip 2pt

\renewcommand*{\proofname}{\textbf{Proof of Theorem~\ref{prop:diagonal}}}
\begin{proof}
Let $G$ be a connected $(k,k-d)$-regular graph, where $d\geq 1$ and $k\geq d^2-d+1$, and let $x$ be a vertex of $G$. 
Since $G_x$ has valency $k-d$,  its complement $\overline{G_x}$ has order $k$ and valency $k-(k-d)-1 = d-1$. 
Also $k\geq d^2-d+1 > (d-1)^2+(d-1)$, and so by Corollary~\ref{cor:far-connected} we find that $\overline{G_x}$ is far-connected,  
and then it follows from Corollary~\ref{cor:CharCompleteMulti} with $t = k-d$ that $d$ divides $k$, and $G$ is isomorphic to the complete multipartite graph $\K_{\frac{k+d}{d}}[\overline{\K_{d}}]$.  
\end{proof}

\section{Small values of $k$ and $t$}\label{sec:smallcases}

Here we give some `sporadic' examples of $(k,t)$-regular graphs for small $k$ and $t$ not covered by our earlier observations, 
followed by a table that summarises our knowledge regarding Question~\ref{BigQuestion} for for $1\leq k \leq 16$ and $0\leq t\leq 14$. 

The sporadic examples include a well-known and obvious graph, 
namely the 1-skeleton of the icosahedron, plus two well-known examples that are strongly regular, 
and some others that are vertex-transitive and easily constructible using elementary group theory.  
The latter were found with the help of the transitive groups database in {\sc Magma} \cite{magma}.

\begin{example}\label{5-2}
There is just one connected $(5,2)$-regular graph: the local graphs must be $\C_5$, and then the only connected example is the $12$-vertex icosahedral graph, by \cite[Proposition 1.1.4]{BCN}. 
\end{example} 

\begin{example}\label{10-6}
The complement of the 5-valent Clebsch graph of order 16 is a connected $(10,6)$-regular graph, 
and in fact is strongly regular with parameters $(16,10,6,6)$; see \cite[p. 104]{BCN}. 
\end{example} 

\begin{example}\label{13-6}
The following construction produces a $(13,6)$-regular graph of order $28$. 
The group $\PSL(2,13)$ has a single conjugacy class of subgroups of order $39$ (isomorphic to the metacyclic group $\C_{13}\rtimes_3 \C_3$), 
and in the natural action of $\PSL(2,13)$ on the $28$ cosets of one of them, the sub-orbits (namely the orbits of the point-stabiliser) 
consist of two of length $1$ and two of length $13$, all of which are self-paired.  
By taking one of the two sub-orbits of length $13$, we can define the edges of an arc-transitive graph that is $(13,6)$-regular, 
with $\PSL(2,13) \times \C_2$ as its automorphism group.
\end{example} 

\begin{example}\label{14-5}
In the natural transitive action of $\Sym(5)$ on the $30$ cosets of a subgroup of order 4 generated by two disjoint transpositions, there are seven sub-orbits of lengths $1, 1, 2, 2, 4, 4$ and $4$ that are self-paired, 
and four of lengths $2, 2, 4$ and $4$ that are not self-paired.  By taking the two non-self-paired orbits of length $4$ and one of the two self-paired orbits of length $2$, 
and a carefully chosen one of the three self-paired orbits of length $4$, we can define the edges of a vertex-transitive graph of order $30$ that is $(14,5)$-regular, 
with $\Sym(5)$ as its automorphism group.
\end{example} 

\begin{example}\label{14-8}
The group $\PGL(2,7)$ has a single conjugacy class of (dihedral) subgroups of order $14$, 
and in its natural action on the $24$ cosets of one of them, 
there are three sub-orbits of length $1$ and three self-paired sub-orbits of length $7$.  
By taking the union of any two of the three sub-orbits of length $7$, we can define the edges of a vertex-transitive graph of order $24$ that is $(14,8)$-regular, 
with automorphism group $\PGL(2,7)$.  In fact, this graph is a Cayley graph for $\Sym(4)$. 
\end{example} 

\begin{example}\label{15-8}
If $a$ and $b$ are generators of orders $2$ and $12$ for the group $D = \C_2 \times \C_{12}$,  
then the complement of the 8-valent Cayley graph for $D$ on the generating set $\{a,b^3,b^4,b^6,ab^4\}$ 
is an arc-transitive graph that is $(15,8)$-regular, with automorphism group of order $17280$. 
\end{example} 

\begin{example}\label{16-10}
The 16-valent Schl\"afli graph of order 27 (naturally associated with the 27 lines on a cubic surface) is a $(16,10)$-regular graph, 
and in fact is strongly regular with parameters $(27,16,10,8)$; see \cite[p.103]{BCN}. 
\end{example} 

We can now present the table, followed by some guidance on how to understand it.
\smallskip

\begin{table}[ht]
\begin{tabular}{ |c|c|c|c|c|c|c|c|c|c|c|c|c|c|c|c| }
\hline
 $(k,t)$ &0& 1 & 2 &3 & 4 & 5&6 & 7 & 8&9 & 10 & 11&12 & 13 & 14\\ \hline
 1 &$\small {\bf \K_2}$& {\color{red}{\color{red}$\times$}} & {\color{red}$\times$} & {\color{red}$\times$} & {\color{red}$\times$} & {\color{red}$\times$} & {\color{red}$\times$} & {\color{red}$\times$} & {\color{red}$\times$} & {\color{red}$\times$} & {\color{red}$\times$} & {\color{red}$\times$} & {\color{red}$\times$} & {\color{red}$\times$} & {\color{red}$\times$} \\\hline
2 &{\small CP}& $\small {\bf \K_3}$ & {\color{red}$\times$} & {\color{red}$\times$} & {\color{red}$\times$} & {\color{red}$\times$} & {\color{red}$\times$} & {\color{red}$\times$} & {\color{red}$\times$} & {\color{red}$\times$} & {\color{red}$\times$} & {\color{red}$\times$} & {\color{red}$\times$} & {\color{red}$\times$} & {\color{red}$\times$} \\\hline
3 &{\small CP}& {\color{red}$\times$} & $\small {\bf \K_4}$ & {\color{red}$\times$} & {\color{red}$\times$} & {\color{red}$\times$} & {\color{red}$\times$} & {\color{red}$\times$} & {\color{red}$\times$} & {\color{red}$\times$} & {\color{red}$\times$} & {\color{red}$\times$} & {\color{red}$\times$} & {\color{red}$\times$} & {\color{red}$\times$} \\\hline
4 &{\small CP}& {\small CP}&{\bf {\small LP}}& $\small {\bf \K_5}$ & {\color{red}$\times$} & {\color{red}$\times$} & {\color{red}$\times$} & {\color{red}$\times$} & {\color{red}$\times$} & {\color{red}$\times$} & {\color{red}$\times$} & {\color{red}$\times$} & {\color{red}$\times$} & {\color{red}$\times$} & {\color{red}$\times$} \\\hline
5 &{\small CP}& {\color{red}$\times$}&{\bf {\color{black}${\ref{5-2}}$}}& {\color{red}$\times$}& $\small {\bf \K_6}$ & {\color{red}$\times$} & {\color{red}$\times$} & {\color{red}$\times$} & {\color{red}$\times$} & {\color{red}$\times$} & {\color{red}$\times$} & {\color{red}$\times$} & {\color{red}$\times$} & {\color{red}$\times$} & {\color{red}$\times$} \\\hline
6 &{\small CP}& {\small CP}&{\small CP}&{\small LP} & {\bf \small LP} & $\small {\bf \K_7}$& {\color{red}$\times$} & {\color{red}$\times$} & {\color{red}$\times$} & {\color{red}$\times$} & {\color{red}$\times$} & {\color{red}$\times$} & {\color{red}$\times$} & {\color{red}$\times$} & {\color{red}$\times$} \\\hline
7 &{\small CP}& {\color{red}$\times$}&{\small CP}& {\color{red}$\times$} &{\color{red}$\times_D$}& {\color{red}$\times$}& $\small {\bf \K_8}$& {\color{red}$\times$} & {\color{red}$\times$} & {\color{red}$\times$} & {\color{red}$\times$} & {\color{red}$\times$} & {\color{red}$\times$} & {\color{red}$\times$} & {\color{red}$\times$} \\\hline
8 &{\small CP}&{\small CP}& {\small CP}&{\small CP}& {\small LP}&{\color{red}$\times_D$}& {\bf {\small LP}}& $\small {\bf \K_9}$ & {\color{red}$\times$} & {\color{red}$\times$} & {\color{red}$\times$} & {\color{red}$\times$} & {\color{red}$\times$} & {\color{red}$\times$}  & {\color{red}$\times$} \\\hline
9 &{\small CP}&{\color{red}$\times$}& {\small CP}&{\color{red}$\times$}& {\small TP}&{\color{red}$\times$}& {\bf \small LP}&{\color{red}$\times$}& $\small {\bf \K_{10}}$& {\color{red}$\times$} & {\color{red}$\times$} & {\color{red}$\times$} & {\color{red}$\times$} & {\color{red}$\times$}  & {\color{red}$\times$}  \\\hline
10 &{\small CP}&{\small CP}& {\small CP}&{\small CP}& {\small CP}&{\small LP}& {\color{black}${\ref{10-6}}$}&{\color{red}$\times_D$}& {\bf \small LP}&$\small {\bf \K_{11}}$& {\color{red}$\times$}& {\color{red}$\times$} & {\color{red}$\times$} & {\color{red}$\times$} & {\color{red}$\times$}   \\\hline
11 &{\small CP}&{\color{red}$\times$}& {\small CP}&{\color{red}$\times$}& {\small CP} &{\color{red}$\times$}& {\color{blue}??}&{\color{red}$\times$}& {\color{red}$\times_D$} &{\color{red}$\times$}&$\small {\bf \K_{12}}$ & {\color{red}$\times$} & {\color{red}$\times$} & {\color{red}$\times$} & {\color{red}$\times$}  \\\hline
12 &{\small CP}&{\small CP}& {\small CP}&{\small CP}&{\small CP}&{\small CP}& {\small LP}&{\color{blue}??}& {\small LP}&{\bf \small LP}& {\bf \small LP}& $\small {\bf \K_{13}}$ & {\color{red}$\times$} & {\color{red}$\times$} & {\color{red}$\times$} \\\hline
13 &{\small CP}&{\color{red}$\times$}& {\small CP}&{\color{red}$\times$}& {\small CP} &{\color{red}$\times$}& {\color{black}${\ref{13-6}}$}&{\color{red}$\times$}& {\color{blue}??} &{\color{red}$\times$}& {\color{red}$\times_D$}& {\color{red}$\times$}& $\small {\bf \K_{14}}$ & {\color{red}$\times$} & {\color{red}$\times$}  \\\hline
14 &{\small CP}&{\small CP}& {\small CP}&{\small CP}& {\small CP} &{\color{black}${\ref{14-5}}$}& {\small CP}&{\small LP}& {\color{black}${\ref{14-8}}$} &{\color{blue}??}& {\color{red}$\times_D$}&{\color{red}$\times_D$}& {\bf \small LP} & $\small {\bf \K_{15}}$& {\color{red}$\times$} \\\hline
15 &{\small CP}&{\color{red}$\times$}& {\small CP}&{\color{red}$\times$}& {\small CP} &{\color{red}$\times$}& {\small CP}&{\color{red}$\times$}& {\color{black}${\ref{15-8}}$} &{\color{red}$\times$}& {\small LP}&{\color{red}$\times$}&{\bf \small LP} & {\color{red}$\times$}&$\small {\bf \K_{16}}$ \\\hline
16 &{\small CP}&{\small CP}& {\small CP}&{\small CP}& {\small CP} &{\small CP}& {\small CP}&{\small CP}& {\small LP} &{\small TP}& {\color{black}${\ref{16-10}}$}&{\color{blue}??}& {\bf \small LP}  & {\color{red}$\times_D$} & {\bf \small LP} \\\hline

\end{tabular}
\medskip
\caption{Small $k$ and $t$}\label{table}
\end{table}

${}$\\[-42pt] 

The rows are indexed by $k$ and the columns are indexed by $t$, and the $(k,t)$th entry gives our best knowledge of the existence or non-existence of a $(k,t)$-regular graph.

An entry marked $\times_D$ or $\times$ indicates non-existence due to Theorem~\ref{prop:diagonal} or Proposition~\ref{Prop:Basic}, respectively, 
while an entry marked CP, LP or TP indicates that an example can be constructed as a non-trivial Cartesian, lexicographic or tensor product, as in Propositions~\ref{Prop:CartesianProduct}, \ref{Prop:LexProduct} or \ref{Prop:TensorProduct}, respectively,  
and an entry such as `5.1' indicates one of the sporadic examples described at the beginning  of this section.  
Next, an entry in bold (including ${\bf K}_n$ for some $n$) indicates that the relevant example is unique amongst connected graphs. 
(In these cases, uniqueness follows from Theorem~\ref{prop:diagonal}, apart from the well-known case of $(k,t)=(5,2)$ in Example~\ref{5-2}.) 
Finally, a pair of question marks indicates that the existence of a $(k,t)$-regular graph is not currently known to us.

\section{Further questions}\label{sec:conclusion}


Finally, we pose some questions that arise naturally from this work:
\smallskip

\noindent
(1) Does there exist a $(11,6)$-regular graph? (This is the smallest open case.) 
\\[+4pt] 
(2) Are Theorem~\ref{prop:diagonal} and Proposition~\ref{Prop:Basic} the only necessary conditions for the existence of \\ ${}$ \quad \ a $(k,t)$-regular graph?
\\[+4pt] 
(3) For which pairs $(k,t)$ is there a unique connected $(k,t)$-regular graph? \\ ${}$ \quad \ (The only such pairs we know of are $(5,2)$ and the ones given by Theorem~\ref{prop:diagonal}.)
\\[+4pt] 
(4) Is there a pair $(k,t)$ such that there is more than one connected $(k,t)$-regular graph, \\ ${}$ \quad \ but every such graph is finite?
\\[+4pt] 
(5) Is there a pair $(k,t)$ such that there exists an infinite $(k,t)$-regular graph but no finite one?

\bigskip

\begin{center} 
{\sc Acknowledgements} 
\end{center} 
\smallskip

We are grateful to the N.Z. Marsden Fund (via grants UOA1824 and UOA2030) and the University of Auckland 
(via FRDF project 3719917) for supporting the research reported in this paper.   
We also acknowledge the help of the {\sc Magma} system \cite{magma} in our search for examples and investigation of their properties.

\end{document}